\documentclass[a4paper,12pt]{amsart}

\usepackage{amssymb}
\usepackage[francais]{babel}
\usepackage{mhequ}
 \usepackage[utf8]{inputenc} 
\renewcommand{\div}{\mbox{div}}

\usepackage{cite}

\DeclareFontFamily{OT1}{rsfs}{}
\DeclareFontShape{OT1}{rsfs}{m}{n}{ <-7> rsfs5 <7-10> rsfs7 <10-> rsfs10}{}
\DeclareMathAlphabet{\mathscr}{OT1}{rsfs}{m}{n}

\newcommand{\bel}[1]{\begin{equation}\label{#1}}
\newcommand{\beal}[1]{\begin{eqnarray}\label{#1}}
\newcommand{\beadl}[1]{\begin{deqarr}\label{#1}}
\newcommand{\eeadl}[1]{\arrlabel{#1}\end{deqarr}}
\newcommand{\eeal}[1]{\label{#1}\end{eqnarray}}
\newcommand{\eead}[1]{\end{deqarr}}
\newcommand{\eea}{\end{eqnarray}}
\newcommand{\eeaa}{\end{eqnarray*}}

\newcommand{\be}{\begin{equation}}
\newcommand{\ee}{\end{equation}}

\DeclareFontFamily{OT1}{rsfs}{}
\DeclareFontShape{OT1}{rsfs}{m}{n}{ <-7> rsfs5 <7-10> rsfs7 <10->
rsfs10}{} \DeclareMathAlphabet{\mycal}{OT1}{rsfs}{m}{n}

\newcounter{mnotecount}[section]

\newcommand{\N}{{\Bbb N}}
\newcommand{\rmnote}[1]{}%{\mnote{#1}}

\newcommand{\Ric}{\operatorname{Ric}}

\def\mysavedown#1{\edef\mysubs{\mysubs#1}}
\def\mysaveup#1{\edef\mysups{\mysups#1}}
\def\mydown#1{{\mytensor}_{\vphantom{\mysubs}#1}}
\def\myup#1{{\mytensor}^{\vphantom{\mysups}#1}}
\def\tensor#1#2{
  #1
  \def\mytensor{\vphantom{#1}}
  \def\mysubs{\relax}
  \def\mysups{\relax}
  \let\down=\mysavedown
  \let\up=\mysaveup
  #2
  \let\down=\mydown
  \let\up=\myup
  #2
  }

\newcommand{\Riem}{\operatorname{Riem}}

\newcommand{\Tr}{\operatorname{Tr}}

\newcommand{\R}{\mathbb R}

\renewcommand{\S}{\mathbb S}

\renewcommand{\div}{\operatorname{div}}

\DeclareMathOperator{\Hess}{Hess}

\renewcommand{\phi}{\varphi}
\renewcommand{\epsilon}{\varepsilon}
\renewcommand{\hat}{\widehat}
\def\crn#1#2{{\vcenter{\vbox{
        \hbox{\kern#2pt \vrule width.#2pt height#1pt
           }
          \hrule height.#2pt}}}}

\newcommand{\Ein}{\operatorname{Ein}}

\renewcommand{\hbar}{{\overline h}}

\newcommand{\pre}[2]{{{\vphantom{#2}}^{#1}}\kern-.2ex{#2}}
\theoremstyle{plain}
\newtheorem{theorem}{Théorème}[section]
\newtheorem{lemma}[theorem]{Lemme}
\newtheorem{proposition}[theorem]{Proposition}

\theoremstyle{definition}

\newtheorem{remark}[theorem]{Remarque}

\numberwithin{equation}{section}

\date{03 avril 2014}

\begin{document}
\title[Inversion d'opérateurs de courbure sur $\R^n$]
{Inversion d'opérateurs de courbures  au voisinage de la métrique euclidienne}

\author[E. Delay]{Erwann Delay}
\address{Erwann Delay,
Labo. de Math. d'Avignon,
 Fac. des Sciences,
33 rue Louis Pasteur, F-84000 Avignon, France}
\email{Erwann.Delay@univ-avignon.fr}
\urladdr{http://www.univ-avignon.fr/fr/recherche/annuaire-chercheurs\newline$\mbox{ }$\hspace{3cm}
/membrestruc/personnel/delay-erwann-1.html}

\begin{abstract}
Nous montrons que certains opérateurs affines en la courbure de Ricci sont localement
inversibles, dans des espaces de Sobolev à poids, au voisinage de la métrique euclidienne.
\end{abstract}

\maketitle

\noindent {\bf Mots clefs } : Courbure de Ricci, 
2-tenseurs symétriques, EDP elliptique quasi-linéaire, espaces de Sobolev à poids.
\\
\newline
\tableofcontents

\section{Introduction}\label{section:intro}
Sur  une variété Riemannienne $(M,g)$, considérons $\Ric(g)$ sa courbure de Ricci  et $R(g)$ sa courbure scalaire.
Parmi les (champs de) 2-tenseurs symétriques géométriques naturels que l'on peut construire,
les plus simples sont ceux qui seront "affines" en la courbure de Ricci, autrement dit, de la forme
$$
\Ein(g):=\Ric(g)+\kappa R(g)g+\Lambda g,
$$
o\`u $\kappa$ et $\Lambda$ sont des constantes.
Ainsi, si $\kappa=\Lambda=0$ on retrouve la courbure de Ricci, si $\kappa=-\frac12$  le tenseur d'Einstein (avec constante cosmologique $\Lambda$), enfin si $\kappa=-\frac1{2(n-1)}$ et $\Lambda=0$ le 
tenseur de Schouten.
Rappelons que ce tenseur est géométriquement  naturel dans le sens o\`u  pour tout difféomorphisme $\varphi$ assez régulier,
$$
\varphi^*\Ein(g)=\Ein(\varphi^*g).
$$
Nous nous posons ici le problème de l'inversion de l'opérateur $\Ein$.
On se donne donc $E$ un champ de tenseur symétrique sur $M$,  on cherche $g$ métrique riemannienne  telle
\bel{mainequation}
\Ein(g)=E.
\ee
On doit ainsi résoudre un système quasi-linéaire particulièrement complexe.
Le cas de la courbure de Ricci prescrite remonte
aux  années 80. 
DeTurck \cite{Deturck:ricci}, en 1981, a tout d'abord montr\'e un r\'esultat
d'existence locale au voisinage d'un point $p$  (il a depuis entrepris une longue étude systématique
pour le cadre local, comme le montrent ses  travaux en 1999
\cite{Deturckrank1}).

Puis il y a eu des  r\'esultats {\it globaux} :  
\cite{Deturckdim2} sur le cas tr\`es particulier de la
dimension 2, pour les surfaces {\it compactes}.  Hamilton \cite{Hamilton1984}, a traité le cas
de la sph\`ere unit\'e de $\R^{n+1}$ (avec $n>2$) en prouvant un
r\'esultat d'inversion locale au voisinage de la m\'etrique
standard.
Nous avions ensuite prouvé un résultat analogue sur l'espace hyperbolique réel \cite{Delay:etude},
et complexe \cite{DelayHerzlich}, au voisinage de la métrique canonique.

Ce type d'inversion locale de l'opérateur de Ricci a été aussi adapté à certaines variétés d'Einstein \cite{DeturckEinstein}, \cite{Delay:study}, \cite{Delanoe2003}.

Notons qu'il existe aussi des r\'esultats d'obstruction sur l'inversion de la courbure de
Ricci  \cite{Deturck-Koiso}, \cite{Baldes1986},
 \cite{Hamilton1984}, \cite{Delanoe1991}, \cite{Delay:etude}.
 
Le but de cet article est de prouver un résultat d'existence locale sur $\R^n$ près de la
métrique euclidienne $\delta$. Nous travaillons pour cela dans des espaces de Sobolev
à poids $H^{s,t}$ de fonctions (ou champs de tenseurs) $u$ telles que 
$\langle x\rangle^{t}u$ est dans l'espace de Sobolev classique $H^s$ (voir section
\ref{sec:Hst} pour une définition plus précise).

\begin{theorem}\label{maintheorem}
Soient  $s,t,\kappa,\Lambda\in\R$ tels que $s>\frac n2$,  $t\geq 0$, $\kappa>-\frac1{2(n-1)}$ et $\Lambda>0$.   Alors pour tout $e\in H^{s+2,t}(\R^n,\mathcal S_2)$ proche de zéro,  il existe $h$ proche de zéro dans 
$H^{s+2,t}(\R^n,\mathcal S_2)$ telle que
$$
\Ein(\delta+h)=\Ein(\delta)+e.
$$
De plus l'application $e\mapsto h$ est lisse au voisinage de zéro entre les espaces de Hilbert correspondants. 
\end{theorem}

Ce résultat permet de traiter un problème similaire à la tentative avortée  de \cite{Jeune1989} pour prescrire la courbure de Ricci,
essentiellement grâce à l'ajout d'une constante cosmologique $\Lambda$ dans l'équation, nous y revenons en section \ref{sec:commentaires}.

Pour un opérateur d'ordre deux, il peut \^etre surprenant de voir que la régularité
de notre solution n'a pas deux points de plus que la donnée. On peut se convaincre
que la régularité est optimale en transposant l'équation par un difféomorphisme peu régulier.

Cette inversion nous permet ensuite, en section \ref{sec:ssvar}
de prouver que l'image de certains opérateurs de type
Riemann-Christoffel sont des sous variétés dans des espaces de Fréchet 
à poids.\\

{\small\sc Remerciements}. {\small Je remercie Philippe Delano\"e de m'avoir signalé en 1994 l'erreur   malheureuse de \cite{Jeune1989} 
qui a fini par motiver ce travail quelques années après...
 Ce projet  est en partie financé par les ANR  SIMI-1-003-01 et ANR-10-BLAN 0105   .

}

\section{Définitions, notations et conventions}\label{sec:def}

Pour une métrique riemannienne $g$, nous noterons  $\nabla$ sa connexion  de Levi-Civita , par $\Ric(g)$   sa courbure de Ricci et par
$\Riem(g)$ sa courbure de  Riemann sectionnelle. 

Soit ${\mathcal T}_p^q$ l'ensemble des tenseurs covariants de rang $p$ et contravariants de rang $q$.
Lorsque $p=2$ et $q=0$, on notera ${\mathcal S}_2$ le sous-ensemble des tenseurs symétriques,
qui se décompose en ${\mathcal G}\oplus {\mathring{\mathcal
S}_2}$ o\`u ${\mathcal G}$ est l'ensemble des  tenseurs  $\delta$-conformes et 
${\mathring{\mathcal S}_2}$ l'ensemble des tenseurs sans trace (relativement à $\delta$). On utilisera la convention de sommation  d'Einstein
 (les indices correspondants vont de $1$ à $n$), et nous utiliserons 
 $g_{ij}$ et son inverse $g^{ij}$ pour monter ou descendre les indices.

Le Laplacien (brut) est défini par
$$
\triangle=-tr\nabla^2=\nabla^*\nabla,
$$
o\`u $\nabla^*$ est l'adjoint formel $L^2$ de $\nabla$. 
Pour  $u$ un 2-tenseur covariant symétrique, on définit sa divergence par
 $$ (\mbox{div}u)_i=-\nabla^ju_{ji}.$$ Pour une 1-forme
$\omega$ on $M$, on définit sa divergence par :
$$
d^*\omega=-\nabla^i\omega_i,
$$
et la partie symétrique  de ses dérivées covariantes:
$$
({\mathcal
L}\omega)_{ij}=\frac{1}{2}(\nabla_i\omega_j+\nabla_j\omega_i),$$
(notons que ${\mathcal L}^*=\mbox{div}$).

On définit l'opérateur de Bianchi des 2-tenseurs symétriques dans les 1-formes :
$$
B_g(h)=\div_gh+\frac{1}{2}d(\Tr_gh).
$$

\section{Espaces à poids et isomorphismes}\label{sec:Hst}
Les espaces à poids que nous utiliserons ici ne sont pas les espaces classiques utilisés dans le contexte
asymptotiquement euclidien de la relativité générale (voir \cite{Bartnikmass}  ou \cite{ChoquetBruhatChristodoulou1981} par exemple ).
Ils sont plut\^ot  utilisés en théorie du scattering comme dans les travaux de S. Agmon \cite{Agmon1975}
ou de R. Melrose  \cite{Melrose1994} (voir aussi son cours \cite{MelroseLec2008} section 6, ou \cite{Schrohe92} p241). 
Pour $x\in\R^n$ , on pose
$$
\langle x\rangle =(1+|x|^2)^{\frac12}.
$$
Pour $s\in\R$, on rappelle  l'espace de Sobolev classique 
$$
H^s(\R^n)=\{u\in\mathcal S'(\R^n)\;,\;\;\langle \xi\rangle^s \hat u\in L^2(\R^n)\},
$$
o\`u $S'(\R^n)$ est le dual topologique de l'espace de Schwartz, et $\hat u$ la transformée de Fourier de $u$. On 
munit $H^s(\R^n)$ de la norme 
$$
\|u\|_s:=\|\langle \xi\rangle^s \hat u\|_{L^2}.
$$
Afin de décrire plus précisément le comportement asymptotique des (champs de) tenseurs
qui nous  intéressent, introduisons les espaces à poids (voir \cite{Agmon1975} ou \cite{MelroseLec2008} par exemple)
$$
H^{s,t}(\R^n)=\{u\in\mathcal S'(\R^n)\;,\;\;\langle x\rangle^t  u\in H^s(\R^n)\}=\langle x\rangle^{-t}H^s(\R^n).
$$
munis de la norme
$$
\|u\|_{s,t}:=\|\langle x\rangle^t u\|_{s}.
$$
Ces espaces ont beaucoup de bonnes propriétés comme :
$\forall s,t\in\R,$ les applications suivantes sont continues
$$
\forall s'\leq s,\; \;\forall t'\leq t\;\;\;H^{s,t}(\R^n)\hookrightarrow H^{s',t'}(\R^n),
$$
$$
\partial_{x_j}:H^{s,t}(\R^n)\longrightarrow H^{s-1,t}(\R^n),
$$
$$
\times x_j: H^{s,t}(\R^n)\longrightarrow H^{s,t-1}(\R^n).
$$
La transformée de Fourier donne un isomorphisme :
$$
\begin{array}{ccc}
H^{s,t}(\R^n)&\widetilde\longrightarrow& H^{t,s}(\R^n)\\
u&\mapsto&\hat u.
\end{array}
$$
Le dual de $H^{s,t}(\R^n)$ s'identifie à $H^{-s,-t}(\R^n)$.
L'injection de Sobolev classique nous donne aussi immédiatement pour $k\in\N$ 
$$
s>\frac n2+k\;\Rightarrow \;\;H^{s,t}(\R^n)\subset \langle x\rangle^{-t}C^k_{\rightarrow 0}(\R^n),
$$
o\`u $C^k_{\rightarrow 0}(\R^n)$ est l'ensemble des fonctions $C^k$ sur $\R^n$ dont les dérivées d'ordre $\leq k$
tendent vers zéro à l'infini.

\begin{proposition}\label{propinv} Pour $s,t\in\R$ et $C>0$ une constante, l'opérateur
$$
\Delta+C:H^{s+2,t}(\R^n)\widetilde\longrightarrow H^{s,t}(\R^n),
$$
est un isomorphisme.
\end{proposition}

\begin{proof}
Par transformée de Fourier, il suffit de remarquer que 
$$
\times(|\xi|^2+C):H^{t,s+2}(\R^n)\longrightarrow H^{t,s}(\R^n)
$$
est un isomorphisme.
\end{proof}
Nous aurons aussi besoin du
\begin{lemma}\label{lemAlgebre}
Soient $s>\frac n2$ et $t\geq0$ alors
$$
u,v\in H^{s,t}\Rightarrow uv\in H^{s,t}.
$$
De plus il existe une constante $C_{s,t}$ telle que
$$
\|uv\|_{s,t}\leq C_{s,t} \|u\|_{s,t} \|v\|_{s,t}.
$$
\end{lemma}
\begin{proof}
Le résultat est connu dans le cas $t=0$ pour une constante $C_s$.
On a ainsi
$$
\|\langle x\rangle^{2t}uv\|_{s}\leq C_{s} \|\langle x\rangle^{t}u\|_{s} \|\langle x\rangle^{t}v\|_{s}.
$$
Ce qui se traduit par  
$$
\|uv\|_{s,2t}\leq C_{s} \|u\|_{s,t} \|v\|_{s,t}.
$$
Il suffit ensuite de rappeler que si $t\geq0$, on a une inclusion continue $H^{s,2t}\subset H^{s,t}$.
\end{proof}

\begin{remark}\label{remespaces}
Les espaces à poids que nous utilisons ici sont,  en un certain sens, plus gros que ceux utilisés habituellement 
dans le contexte asymptotiquement euclidien de la relativité générale comme dans \cite{ChoquetBruhatChristodoulou1981} ou \cite{Bartnikmass}.
En effet nous ne demandons pas ici aux dérivées de décroître plus vite (ou croître moins vite) que la fonction à l'infini.
\end{remark}

\section{Preuve du théorème principal}
Il est maintenant bien connu que l'équation que nous voulons résoudre (\ref{mainequation}) n'est pas elliptique dû à l'invariance
de la courbure par difféomorphisme. Nous allons modifier cette  équation en s'inspirant
de la m\'ethode de DeTurck. On  y ajoute donc un terme jauge de telle sorte
que le cette nouvelle équation devienne elliptique, tout en faisant en sorte
que ses solutions soient encore solutions de l'équation de départ.

Tout d'abord comme 
$$
\Tr_g\Ein(g)=(1+n\kappa)R(g)+n\Lambda,
$$
l'équation (\ref{mainequation}) est équivalente à 
$$
\Ric(g)=E-\frac{\kappa\Tr_g  E+\Lambda}{1+n\kappa}g.
$$ 
Pour toute métrique $g$, $B_{g}(\Ric(g))=0$
par l'identité de Bianchi. Nous définissons donc 
$$
\mathcal B_g(E)=\div_gE+\frac{2\kappa+1}{2(1+\kappa n)}d\Tr_gE=B_g(E)-\frac{(n-2)\kappa}{2(1+\kappa n)}d\Tr_gE,
$$
de sorte que l'identité de Bianchi se traduise ici par
$$
\mathcal B_g(Ein(g))=0.
$$

Afin de construire notre nouvelle équation, rappelons quelques différentielles d'opérateurs.
On a d'une part (voir \cite{Besse} par exemple)
$$
D\Ric(\delta)h=\frac12\Delta h-\mathcal L_{\delta}B_{\delta}(h).
$$
D'autre part, compte tenu de la différentielle de $B_g(E)$ relativement à la métrique (voir \cite{Delay:study} par exemple),
on trouve
$$
D[\mathcal B_{(.)}(E)](\delta)h=-EB_{\delta}(h)+\frac{(n-2)\kappa}{2(1+\kappa n)}d\langle E,h\rangle
+T(E,h),$$
o\`u $E$ est identifié à l'endomorphisme de $T^*M$ correspondant et
$$
T(E,h)_j=\frac12(\partial_kE_{jl}+\partial_lE_{kj}-\partial_jE_{kl})h^{kl},
$$
en particulier $T(E,h)=0$ si $E$ est proportionnel à $\delta$.
On définit, pour l'instant formellement, pour $\kappa\neq -1/n$, $\Lambda\neq0$, $h$ et $e$ voisins de zéro dans $H^{s+2,t}(\R^n,\mathcal S_2)$:
$$
\mathcal F(h,e):=\Ric(\delta+h)-E+\frac{\kappa\Tr_{\delta+h}E+\Lambda}{1+\kappa n}{(\delta+h)}-\frac1\Lambda\mathcal L_{\delta}\mathcal B_{\delta+h}(E),
$$
o\`u $E=\Ein(\delta)+e=\Lambda\delta+e$.
\begin{proposition}\label{Flisse}
Pour $\kappa\neq -1/n$, $\Lambda\neq0$, $s>\frac n2$ et $t\geq 0$ l'application $$
\mathcal F: H^{s+2,t}(\R^n,\mathcal S_2)\times H^{s+2,t}(\R^n,\mathcal S_2)\longrightarrow H^{s,t}(\R^n,\mathcal S_2),
$$
est bien définie et lisse au voisinage de zéro.
\end{proposition}
\begin{proof}
La preuve de cette proposition est renvoyée en appendice,
elle utilise essentiellement le fait  que sous ces hypothèses,
l'espace $H^{s,t}$ est une algèbre.
\end{proof}

\begin{proposition}\label{propsolF}
Soient $s>\frac n2$, $t\geq 0$, $\Lambda>0$ et $k>-\frac1{2(n-1)}$. Pour tout $e$ assez petit dans $H^{s+2,t}(\R^n,\mathcal S_2)$,
il existe $h$ dans $H^{s+2,t}(\R^n,\mathcal S_2)$ tel que $\mathcal F(h,e)=0$,
de plus l'application $e\mapsto h$ est lisse entre les espaces de Hilbert
correspondants.
\end{proposition}
\begin{proof}
On a déjà
$$
\mathcal F(0,0)=0,
$$
et la différentielle de $\mathcal F$  relativement à la première variable est
$$
D_h\mathcal F(0,0)h=
\frac12\Delta h+\Lambda h-\frac{\kappa\Lambda}{1+\kappa n}
 \Tr_\delta h \;\delta
-\frac{(n-2)\kappa}{2(1+\kappa n)}\partial\partial\Tr_\delta h.
$$

On remarque que lorsque $\kappa\neq0$, $D_h\mathcal F(0,0)$ ne préserve pas le scindage $\mathcal G\oplus \mathring S_2$.
En effet dans la direction  conforme $h=u\delta$ on trouve 
$$
D_h\mathcal F(0,0)(u\delta)=\frac1{2(1+\kappa n)}[(1+2(n-1)\kappa)\Delta u+2\Lambda u ]\delta
-\frac{(n-2)n\kappa}{2(1+\kappa n)}\mathring \Hess\; u,
$$
o\`u $\mathring \Hess \;u$ est la partie sans trace de la hessienne de $u$. 
En revanche, dans la direction  $h=\mathring h$  sans trace, on a
$$
D_h\mathcal F(0,0)\mathring h=\frac12(\Delta+2\Lambda)\mathring h.
$$
Quoiqu'il en soit, compte tenu de la proposition \ref{propinv}, si $\Lambda>0$ et $\kappa>-\frac1{2(n-1)}$,  l'opérateur $D_h\mathcal F(0,0)$ 
est un  isomorphisme de $H^{s+2,t}(\R^n,\mathcal S_2)$ dans $H^{s,t}(\R^n,\mathcal S_2)$.
Le théorème des fonctions implicite permet alors de conclure.
\end{proof}
\begin{remark}
On  voit apparaître naturellement la constante critique $\kappa=-1/2(n-1)$ correspondant au tenseur de Schouten 
dans la direction conforme.
\end{remark}
\begin{proposition}\label{proposolEin}
Sous les hypothèses de la proposition \ref{propsolF}, quitte à réduire les voisinages de zéro,  $h$
est solution de 
$$\Ein(\delta+h)=E.$$
\end{proposition}
\begin{proof}
On applique  $B_{\delta+h}$ à l'équation  $\mathcal F(h,e)=0$, ainsi
$$
B_{\delta+h}\mathcal F(h,e)=-\mathcal B_{\delta+h}(E)-\frac1\Lambda B_{\delta+h}\mathcal L_{\delta}\mathcal B_{\delta+h}(E)=0.
$$
On pose $\omega=\frac1\Lambda\mathcal B_{\delta+h}(E)$ alors
$$
P_{\delta+h}\omega :=B_{\delta+h}\mathcal L_{\delta}\omega+\Lambda\omega=0,
$$
avec,  comme il est justifié en appendice, $\omega\in H^{s+1,t}(\R^n,\mathcal T_1)$. 
Or par la proposition \ref{propinv}, comme 
$$
P_\delta=\frac12(\Delta +2\Lambda )
$$
est un isomorphisme de $H^{s+1,t}(\R^n,\mathcal T_1)$ dans $H^{s-1,t}(\R^n,\mathcal T_1)$, l'opérateur $P_{\delta+h}$ 
reste injectif dans le m\^eme espace si $h$ est assez petit dans $H^{s+2,t}\subset H^{s,0}$.
On obtient finalement
$$
\omega=0
$$
\end{proof}

\section{Image d'opérateurs de courbures de type Riemann-Christoffel}\label{sec:ssvar}

Nous voudrions, tout comme dans \cite{Delay:etude}, montrer que l'image de certain opérateurs de 
type Riemann-Christoffel, sont des sous variétés dans $C^\infty$, au voisinage de la  métrique euclidienne $\delta$.
Nous cherchons donc tout d'abord un tenseur  $\mathcal Ein$ qui soit  4 fois covariant, ayant les m\^emes propriétés 
algébriques que le tenseur de Riemann et affine en la courbure, on 
pose donc 
$$
\mathcal Ein(g)=\Riem(g)+g {~\wedge \!\!\!\!\!\bigcirc ~} (a\Ric(g)+bR(g)g+cg), 
$$
o\`u ${~\wedge \!\!\!\!\!\bigcirc ~}$ est le produit de Kulkarni-Nomizu (\cite{Besse} p. 47).
Comme nous voulons que $\Tr_g\mathcal Ein(g)$ soit proportionnelle à $\Ein(g)$, cela nous impose 
 $$c=\frac{1+(n-2)a}{2(n-1)}\Lambda,\;\;b=\frac{\kappa[1+a(n-2)]-a}{2(n-1)}.$$

On a alors
$$
\Tr_g\mathcal Ein(g)=[a(n-2)+1]\Ein(g).
$$
Nous définirons la version de type Riemann-Christoffel de $\mathcal Ein(g)$ par
$$
[g^{-1}\mathcal Ein(g)]^i_{klm}:=g^{ij}\mathcal Ein(g)_{jklm}.
$$
Consid\'erons ${\mathcal R}^1_3$, le sous-espace de ${\mathcal T}^1_3$ des
tenseurs v\'erifiants
$$
\tau^i_{ilm}=0,\;\tau^i_{klm}=-\tau^i_{kml},\;
\tau^i_{klm}+\tau^i_{mkl}+\tau^i_{lmk}=0.
$$
On définit l'espace de Fréchet 
$$C^{\infty,t}=\cap_{k\in\N}H^{k,t},$$
munit de la famille de semi-normes  $\{\|.\|_{k,t}\}_{k\in\N}$.
On procède alors de façons similaire à \cite{Delay:etude} pour prouver que 
\begin{theorem}
Sous les conditions du théorème \ref{maintheorem},  l'image de l'application
$$
\begin{array}{lll}
C^{\infty,t}(\R^n, \mathcal S_2)&\longrightarrow&C^{\infty,t}(\R^n,\mathcal R_3^1)\\
h&\mapsto &(\delta+h)^{-1}\mathcal Ein(\delta+h)-(\delta)^{-1}\mathcal Ein(\delta)\\
\end{array}
$$
est une sous-variété lisse au voisinage de zéro.
\end{theorem}

\begin{remark}
Dans la définition de $\mathcal Ein$, le choix de $a\neq -1/2(n-1)$ est  encore libre.
Si nous  voulions retrouver la courbure de Riemann lorsque
$\kappa=\Lambda=0$ et, lorsque  $\kappa=-1/2$, un tenseur à divergence nulle (donc $a=-1$ et $b=1/4$ via l'identité de Bianchi 2), 
on pourrait choisir  par exemple  
$$
a=2\kappa\;,\;\;\;\; b=\frac{\kappa[2\kappa(n-2)-1]}{2(n-1)}\;,\;\;\;\;c=\frac{\Lambda[2\kappa(n-2)+1]}{2(n-1)}.
$$
Il n'est pas clair que ce choix soit plus naturel qu'un autre,
peut être qu'une identité de type Bianchi 2 qui en découlerait serait aussi plus légitime mais
nous n'avons pas pu trancher à ce stade.

\end{remark}

\section{Commentaires et perspectives}\label{sec:commentaires}

Comme signalé en introduction le résultat pour le cas $\kappa=\Lambda=0$ est annoncé dans \cite{Jeune1989} mais la démonstration comporte une erreur
dans la preuve de la  proposition page 363 signalée par Philippe Delano\"e.  En effet, avec les notations de cette note au C.R.A.S., m\^eme si $L_R$ est surjective dans les bon espaces
à poids , $R.L_R$ n'est plus surjective dans ces m\^eme  espaces ainsi $f'(e)$ ne l'est pas non plus. Ce problème  est dû
au comportement asymptotique de $R$ qui, m\^eme  en supposant  $R$ inversible,  tend vers zéro à l'infini, en particulier $R$ n'est pas inversible  à l'infini.
M\^eme si nous n'avons pas cherché de contre-exemple au théorème principal de \cite{Jeune1989}, le résultat annoncé semble ainsi peu probable.

Le travail présenté ici remédie d'une certaine manière à ce problème par l'ajout d'une constante cosmologique $\Lambda>0$,
ainsi en particulier, à l'infini, $E=\Lambda\delta$  est encore inversible.

\medskip

Il serait intéressant d'étudier un résultat analogue sur une variété asymptotiquement euclidienne en un sens approprié. Notons ici que la définition naturelle n'est pas  celle utilisée habituellement dans ce type de contexte (voir remarque \ref{remespaces}) et probablement  que la définition adaptée est celle de \cite{Melrose1994}.
Nous approfondirons  cette direction dans un futur proche.
%\begin{remark}
%Il pourait etre tentant d'adapter le résultat en remlacant la métrique euclidienne par une métrique AE et Ricci parallèle (voir \cite{DelayRicPar}).
%Mais une analyse simple permet de montrer qu'alors l'espace en question est l'espace euclidien (laplacien ricci=0 donc ricci=0 donc g AE a un grand ordre
%et finalement g est la metrique euclid ) voir Bartnik mass une proposition sur ricci, on peut aussi construire n Kiling asympto constant .\erw{c est pas clair  si on a droit au poids 0 car  c'est plus vraiment AE...}
%\end{remark}

\medskip

L'inversion de ce type d'opérateur doit pouvoir aussi être réalisée au voisinage d'autres modèles non compactes
à courbure de Ricci parallèle comme $\S^k\times\R^{n-k}$. Il faudra alors   s'inspirer de \cite{Delay:ricciproduit}. Ce sera aussi l'objet de futurs travaux. 
\section{Appendice}
Nous justifions ici la proposition \ref{Flisse} par une preuve relativement formelle.
Nous renvoyons le lecteur encore sceptique à  \cite{Delay:etude} o\`u une preuve similaire 
est particulièrement détaillée.

Rappelons que la courbure de Ricci s'exprime en coordonnées locales par
$${\small
\Ric(g)_{jk}=\partial_l\Gamma^l_{jk}-\partial_k\Gamma^l_{jl}}
+\Gamma^p_{jk}\Gamma^l_{pl}-\Gamma^p_{jl}\Gamma^l_{pk},
$$
o\`u
$$
\Gamma^k_{ij}=\frac{1}{2}g^{ks}(\partial_ig_{sj}+\partial_jg_{is}
-\partial_sg_{ij}).
$$
Nous écrirons donc abusivement
$$
\Ric(g)=\partial\Gamma+\Gamma\Gamma\;,\;\;\; \Gamma=g^{-1}\partial g.
$$
Ici nous avons $g=\delta+h$ avec $h$ petit dans $ H^{s+2,t}$, $s>\frac n2$, $t\geq 0$. On a alors
$$g^{-1}=\delta^{-1}+\widetilde h\;,\;\;\widetilde h\in H^{s+2,t},$$
avec par inégalité triangulaire et par le lemme \ref{lemAlgebre}
$$
\|\widetilde h\|_{s+2,t}\leq \sum_{k\in\N}C_{s+2,t}^{k}\|h\|_{s+2,t}^{k+1}=\frac{\|h\|_{s+2,t}}{1-C_{s+2,t}\|h\|_{s+2,t}}.
$$
Pour $\|h\|_{s+2,t}\leq \frac1{2C_{s+2,t}}$, ce qu'on suppose désormais, on a 
$$
\|\widetilde h\|_{s+2,t}\leq 2\|h\|_{s+2,t}.
$$
On obtient alors, en utilisant encore le lemme \ref{lemAlgebre}, et en omettant dorénavant les constantes 
$$
\Gamma=(\delta^{-1}+\widetilde h)\partial h\in H^{s+1,t}\;,\;\;\|\Gamma\|_{s+1,t}\leq \|h\|_{s+2,t},
$$
et 
$$
\partial \Gamma\in H^{s,t}\;,\;\;\|\partial\Gamma\|_{s,t}\leq \|\Gamma\|_{s+1,t}\leq \|h\|_{s+2,t},
$$
d'o\`u, toujours par le lemme \ref{lemAlgebre},
$$
\Ric(g)\in H^{s,t}\;,\;\;\|\Ric(g)\|_{s,t}\leq \|h\|_{s+2,t}.
$$
Étudions maintenant l'opérateur de Bianchi 
$$
\mathcal B_g(E)=\div_gE+\frac{2\kappa+1}{2(1+\kappa n)}d\Tr_gE,
$$
que nous écrirons encore abusivement 
$$
\mathcal B_g(E)=g^{-1}(\partial  E+\Gamma E)+\partial (g^{-1}E).
$$
Compte tenu des calculs précédent et du fait que $E=\Lambda\delta+e$,
on a 
$$
\mathcal B_g(E)=(\delta^{-1}+\widetilde h)[\partial  e+\Gamma (\Lambda\delta+e)]+\partial [\delta^{-1}e+\widetilde h (\Lambda\delta+e)].
$$
On estime alors comme précédemment  
$$
\mathcal B_g(E)\in H^{s+1,t}\;,\;\;\|\mathcal B_g(E)\|_{s+1,t}\leq (\|h\|_{s+2,t}+\|e\|_{s+2,t}),
$$
et
$$
\mathcal L_\delta\mathcal B_g(E)\in H^{s,t}\;,\;\;\|\mathcal L_\delta\mathcal B_g(E)\|_{s,t}\leq
\|\mathcal B_g(E)\|_{s+1,t}\leq (\|h\|_{s+2,t}+\|e\|_{s+2,t}).
$$
Il reste a estimer le terme d'ordre zéro :
$$
Z:=\frac{\kappa\Tr_{\delta+h}E+\Lambda}{1+\kappa n}{(\delta+h)}-E
$$
On écrit encore formellement, en se souvenant ici que le premier "produit" est une trace,
$$
\begin{array}{lll}
(1+n\kappa)Z&=&[\kappa(\delta^{-1}+\widetilde h)(\Lambda \delta+e)+\Lambda](\delta+h)-(1+n\kappa)(\Lambda \delta+e)\\
&=&[\kappa\delta^{-1}e+\kappa\widetilde h(\Lambda \delta+e)+(1+n\kappa)\Lambda](\delta+h)-(1+n\kappa)(\Lambda \delta+e).\\
\end{array}
$$
En développant, on remarque que le terme constant est nul et que l'on peut estimer
comme auparavant, pour $k\neq -1/n$,
$$
Z\in H^{s,t}\;,\;\;\|Z\|_{s,t}\leq \|Z\|_{s+2,t}\leq(\|h\|_{s+2,t}+\|e\|_{s+2,t}).
$$
\bibliographystyle{amsplain}

\bibliography{../references/newbiblio,%
../references/reffile,%
../references/bibl,%
../references/hip_bib,%
../references/newbib,%
../references/PDE,%
../references/netbiblio,%
../references/erwbiblio,%
 stationary}

\end{document}